\documentclass[a4paper,11pt]{article}
\usepackage{amssymb, epsf, lscape, color, colortbl, subfigure, epsfig, setspace,graphicx}
\usepackage{amsmath, amsthm, fancybox, float, amsfonts,url}
\usepackage{xcolor, soul}

\newtheorem{lemma}{Lemma}

\newtheorem{prop}{Proposition}
\newtheorem{thrm}{Theorem}

\newcommand{\Diag}{\mathrm{Diag} }

\newcommand{\dst}{distribution }

\newcommand{\flw}{following }

\begin{document}

\author{Michael Grabchak\footnote{Email: mgrabcha@charlotte.edu}, Xingjie Li, and Isaac M. Sonin \\
University of North Carolina Charlotte}
\title{The Test and Find Model}
\date{\today}
\maketitle

\begin{abstract}
We introduce the Test and Find (TF) problem, where a decision maker (DM) faces the \flw situation: $k$ identical objects are randomly allocated to $n$ distinct boxes (sites) according to some distribution $\pi$, with no more than one object to a box. The DM tests all of the boxes. However, the tests are imperfect: they can give false positive or false negative results. DM has $m$ tags, $1\leq m\leq n$, and, after testing all boxes, she can place a tag on any box that she thinks has a hidden object. She is rewarded $c_i$ for a correct guess and penalized $d_i$ for a wrong guess in box $i$. DM knows all of the parameters of the model and her goal is to maximize the \emph{expected reward}. We give an explicit solution to this problem. We then turn to the symmetric case, for which we derive more computationally efficient results. We also consider several extensions of the TF model and give detailed solutions. One of these extensions is to the realistic case where $k$, the number of objects, is unknown and random.  \\
	
\end{abstract}

\section{Introduction}\label{sec: INT}

We introduce the following problem: $k$ identical objects are randomly allocated to $n$ distinct boxes (or sites) according to some \dst $\pi$, no more than one object to a box. To begin with, we assume that $k$ is known and fixed, but later we extend to the case where it may be random. A decision maker (DM) tests all of the boxes. However, the tests are imperfect: they can give false positive or false negative results. We assume that tests are performed independently for each box. The DM has $m$ tags, $1\leq m\leq n$, and, after testing all boxes, she can place a tag on any box that she thinks has a hidden object. She is allowed to use as many or as few of the tags as she likes. She is rewarded $c_i$ for a correct guess and penalized $d_i$ for a wrong guess in box $i$. The DM knows all of the parameters of the model and her goal is to maximize the \emph{expected reward}. We call this the Test and Find (TF) problem. A few potential applications are as follows:
\begin{itemize}
\item A person has $k\ge1$ unknown diseases from a set of $n$ possible diseases. We give cheap (but not very effective) tests for each disease.  We can afford at most $m$ expensive tests and we must select which of these to administer.
\item There are $n$ people, $k\ge 1$ of whom have a certain disease. We screen everybody with a cheap test and we can select at most $m$ people to give expensive tests to.
\item A system (e.g.,\ a network or an electronic component) has $n$ blocks, of which $k\ge1$ are defective. We run preliminary diagnostics on each block. Based on these results, we decide which blocks to examine in more detail.
\item There are $n$ locations, $k\ge1$ of which have a hazard, e.g., a bomb or a group of terrorists. We can send resources, e.g., first responders, bomb disposal specialists, or drones, to $m$ sites. Here testing may represents surveillance or intelligence.
\end{itemize}

We give a complete solution to the TF problem. While the solution is simple to state, it can be computationally difficult to evaluate in large-scale situations. For this reason we give additional focus to a symmetric case, for which we derive a computationally efficient solution. In this symmetric case, we assume that objects are allocated to boxes following a uniform distribution and that all boxes have the same parameters of estimation. However, the cost structure (i.e.,\ the rewards and penalties) is allowed to differ from box to box. We also consider three extensions of the TF problem. The first is to the case where we get a reward only when we correctly identify all of the objects without making any mistakes. In the second we are still not allowed to make mistakes, but now we only need to identify $m\le k$ of the boxes with objects. In the third, we allow for the number of objects $k$ to be random.

It is not easy to classify the TF problem. It can be seen as a problem from statistical decision theory or a problem from search theory in the spirit of the books \cite{Alp13} and \cite{Gar00} and the papers \cite{Staroverov:1963} and \cite{Cla23}, or just as a problem from mathematical statistics. It is closely related to the Locks, Bombs, and Testing (LBT) problem, which was introduced by Isaac Sonin, one of the authors of the current paper, in the series of papers \cite{lison21}, \cite{son24}, and \cite{son25}. Subsequent work on this problem by other researchers can be found in \cite{Liu:Mazlov:2026}, \cite{Liu:Mazlov:2026b}, and \cite{Schreider:2026}. The LBT problem has $n$ boxes, an Attacker with $m$ ``bombs" that can destroy the boxes, and a Defender with $k$ ``locks" that can be placed in a box to prevent its destruction. This is an adversarial two-player game. In contrast, the TF problem can be seen as a non-adversarial one-player game. While some similar tools are used to solve both problems, they are quite different and aim to model very different applications.

The rest of this paper is organized as follows. In Section~\ref{sec: GTF} we give a detailed solution to the general TF problem. In Section~\ref{sec: ident test} we specialize our results to the case where the testing parameters are the same in each box. In Section~\ref{sec: Example} we present an illustrative  example to better explain  the solution and to show the kinds of situations that one can encounter. In Section~\ref{sec: STF} we give results for the symmetric case and in Section~\ref{sec: Extensions} we give results for three interesting extensions. Finally, we give some conclusions and directions for future work in Section~\ref{sec:con}. 

\section{The General Test and Find Model}\label{sec: GTF}

In this section we give a complete solution to the general TF problem with known $k$. We begin by introducing some notation. Let $T_{i}$ and $S_{i}$ for $i=1,2,\dots,n$ be random variables, each taking only two values: $0$ and $1$. $T_{i}=1$ when the $i$th box has an object and $T_i=0$ otherwise; $S_{i}=1$ when the $i$th box test is positive (i.e., when the test in box $i$ says that there is an object in that box) and $S_i=0$ otherwise. We assume that the probabilities of correct identification of both types, in statistical language the \emph{sensitivity} and \emph{specificity}, are known and given by: $P(S_i=1|T_i=1)=a_i$ and $P(S_i=0|T_i=0)=b_i$. Throughout we assume that $a_i,b_i\ge1/2$. We denote $S=(S_1,S_2,\dots,S_n)$ and $T=(T_1,T_2,\dots,T_n)$. Let $\Gamma = \Gamma_{n,k}$ and $\Sigma=\Sigma_{n,k}$ represent, respectively, the set of all possible realizations of $T$ and the set of all possible realizations of $S$. Note that $\gamma\in\Gamma$ means that $\gamma=(\gamma_{1},\gamma_{2},\dots,\gamma_{n})$ with $\gamma_j\in\{0,1\}$ and $\sum_{j=1}^n \gamma_j=k$ and $s\in\Sigma$ means that $s=(s_{1},s_{2},\dots,s_{n})$ with $s_j\in\{0,1\}$. It can be readily checked that $|\Gamma|= \binom{n}{k}=:M$ and $|\Sigma|= 2^n=:L$, where $|\cdot|$ denotes the cardinality of a set. When we observe $S=s$ for some $s\in\Sigma$, we call $s$ the \emph{signal}, as it is the signal received from the testing. The allocation of objects and the subsequent testing may be viewed as a two-stage random experiment with outcomes represented by pairs $(\gamma ,s)\in\Gamma\times\Sigma=:\Omega$. Note that the DM only observes the second component $s$ of an outcome $\omega=(\gamma,s)\in\Omega$.

Throughout, we assume that tests are performed independently and that the result of any test depends only on whether there is an object in the corresponding box or not. Mathematically, this can be formalized by the following two assumptions. First, we assume that the random variables $S_1,S_2,\dots,S_n$ are conditionally independent given $T$. Second, we assume that for any $\gamma\in\Gamma$ we have
\begin{eqnarray}\label{eq: cond on T}
P(S_i=s_i |T=\gamma) = P(S_i=s_i|T_i = \gamma_i).
\end{eqnarray}
These assumptions imply the following facts.

\begin{lemma}\label{lemma: cond indep}
1. For any $\gamma\in\Gamma$ and $s\in\Sigma$, we have
$$
P(S=s|T=\gamma) = \prod_{i=1}^n P(S_i=s_i|T_i=\gamma_i).
$$
2. Fix $s\in\Sigma$ and $\gamma_i\in\{0,1\}$ and let $S_{-i}$ and $s_{-i}$ denote, respectively, the vectors $S$ and $s$ with the $i$th coordinate removed. We have
$$
P(S_{-i}=s_{-i}, S_i=s_i|T_i=\gamma_i) = P(S_{-i}=s_{-i}|T_i=\gamma_i) P(S_i=s_i|T_i=\gamma_i).
$$
\end{lemma}

\begin{proof}
The first part follows immediately from the fact that the $S_i$'s are conditionally independent given $T$ and \eqref{eq: cond on T}. To see the second part, note that
\begin{eqnarray*}
P(S_{-i}=s_{-i}, S_i=s_i|T_i=\gamma_i) &=& P(S_{-i}=s_{-i}, S_i=s_i,T_i=\gamma_i) \frac{1}{P(T_i=\gamma_i) }\\
&=& \sum_{\gamma'\in\Gamma:\gamma'_i=\gamma_i} P(S_{-i}=s_{-i}, S_i=s_i, T=\gamma') \frac{1}{P(T_i=\gamma_i) } \\
&=& \sum_{\gamma'\in\Gamma:\gamma'_i=\gamma_i} P(S_{-i}=s_{-i}, S_i=s_i| T=\gamma') \frac{P(T=\gamma')}{P(T_i=\gamma_i)} \\
&=& \sum_{\gamma'\in\Gamma:\gamma'_i=\gamma_i} P(S_{-i}=s_{-i}| T=\gamma') P( S_i=s_i| T=\gamma') \frac{P(T=\gamma')}{P(T_i=\gamma_i)} \\
&=&  \frac{P( S_i=s_i| T_i=\gamma_i)}{P(T_i=\gamma_i)}  \sum_{\gamma'\in\Gamma:\gamma'_i=\gamma_i} P(S_{-i}=s_{-i}, T=\gamma') \\
&=& P(S_{-i}=s_{-i}|T_i=\gamma_i) P(S_i=s_i|T_i=\gamma_i),
\end{eqnarray*}
where the fourth line follows by the conditional independence of the $S_i$'s and fifth by \eqref{eq: cond on T}.
\end{proof}

Next, we turn to the joint distribution of $S$ and $T$. Any probability measure $\pi$ on $\Gamma$ can serve as the distribution of $T$. Given distribution $\pi$, let $p(s|\gamma)$, $p(\gamma ,s|\pi)$, $p(s|\pi)$, and $\theta(\gamma |s,\pi)$ denote, respectively, the conditional  probability mass function (pmf) of $S$ given $T=\gamma$, the joint pmf of $T$ and $S$, the marginal pmf of $S$, and the conditional pmf of $T$ given $S=s$. The notation emphasizes the dependence on $\pi$, except for $p(s|\gamma)$, which does not depend on $\pi$. We can think of $\pi$ as the prior distribution of $T$ and $\theta(\gamma |s,\pi)$ as the posterior distribution, given that signal $S=s$ is observed; we use Greek letters to emphasize this relationship. We also define $\alpha_i(s,\pi)= P(T_i=0 |S=s,\pi)$ for $i=1,2,\dots,n$, which will be important going forward. We now characterize these quantities.

\begin{prop}\label{Prop1} If $\gamma=(\gamma_1,\dots,\gamma_n) \in \Gamma$ and $s=(s_1,\dots,s_n)\in\Sigma$, then for any $\pi$
\begin{eqnarray}
p(s|\gamma)&=&\prod_{i: \gamma_i=1}a_i^{s_i}(1-a_i)^{1-s_i}\prod_{i: \gamma_i=0}b_i^{1-s_i}(1-b_i)^{s_i}, \label{eq: psg}\\
p(\gamma,s|\pi)&=&\pi(\gamma)p(s|\gamma),       \label{psgg}\\
p(s|\pi)&=&\sum_{\gamma\in \Gamma}\pi(\gamma)p(s|\gamma),       \label{pspi}\\
\theta(\gamma|s,\pi) &=&\frac{\pi(\gamma)p(s|\gamma)}{p(s|\pi)} , \label{pgs}\\
\alpha_i(s,\pi)&=& \sum_{\gamma\in\Gamma: \gamma_i=0}\theta(\gamma |s,\pi)=
\frac{\sum\limits_{\gamma\in\Gamma: \gamma_i=0} \pi(\gamma)p(s|\gamma)}{\sum_{\gamma\in \Gamma}\pi(\gamma)p(s|\gamma)}, \mbox{ and }\label{eq: alpha i}\\
\sum_{i=1}^n\alpha_i(s,\pi) &=& n-k. \label{snk}
\end{eqnarray}
\end{prop}

\begin{proof}
Since the $S_i$'s are conditionally independent given $T=\gamma$,
$$
p(s|\gamma)=\prod_{i=1}^{n}p(s_i|\gamma)=\prod_{i: \gamma_i=1}p(s_i|\gamma)\prod_{i: \gamma_i=0}p(s_i|\gamma).
$$
Furthermore, given $T=\gamma$, the random variable $S_i$ has a Bernoulli distribution with pmf $p(x)=p^x(1-p)^{1-x}$ for $x=0,1$, where $p=a_i$ if $\gamma_i=1$ and $p=1-b_i$ if $\gamma_i=0$. From here \eqref{eq: psg} follows directly. Next, \eqref{psgg} follows from the definition of conditional probability, \eqref{pspi} from the law of total probability, \eqref{pgs} from Bayes' theorem, and \eqref{eq: alpha i} from the fact that $\alpha_i(s,\pi)= P(T_i=0 |S=s,\pi)=\sum_{\gamma\in\Gamma: \gamma_i=0}P(T=\gamma|S=s,\pi)$ combined with \eqref{pgs} and \eqref{pspi}. Finally, to show \eqref{snk}, let $1_{[T_i=0]}$ denote the indicator function on event $[T_i=0]$ and note that $\sum_{i=1}^n 1_{[T_i=0]}$ is the number of empty boxes, which is $n-k$. From the definition of $\alpha_i(s,\pi)$ to have
\[
\begin{aligned}
\sum_{i=1}^n\alpha_i(s,\pi)=&
\sum_{i=1}^n P(T_i=0 |S=s,\pi)\\
=&\mathrm E\left[\sum_{i=1}^n 1_{[T_i=0]}\Big{|}S= s,\pi\right] = \mathrm E\left[ n-k\Big{|} S=s,\pi\right] = n-k,
\end{aligned}
\]
which completes the proof.
\end{proof}

We are now ready to describe the solution to the TF problem. The basic idea is relatively simple. After performing the testing and observing signal $s$, DM uses \eqref{eq: alpha i} to calculate the probabilities $\alpha_i(s,\pi)=P(T_i=0|S=s,\pi)$ for each $i$. Note that $\alpha_i(s,\pi)$ represents the probability that the $i$th box is empty, given the results of all of the tests, i.e., given that we observed signal $s$. Next, letting $R_i(s,\pi)$ be the expected reward from site $i$ if it is tagged, we have
\begin{eqnarray}
	 R_i(s,\pi)=c_iP(T_i=1|S=s,\pi)-d_iP(T_i=0|S=s,\pi)=c_i-(c_i+d_i)\alpha_i(s,\pi). \label{Ri}
\end{eqnarray}
It follows that $R_i(s,\pi)>0$ if and only if $\frac{c_i}{c_i+d_i}>\alpha_i(s,\pi)$. From here, the optimal strategy for the DM is clear: Let $\ell$ be the number of sites $i$ with $R_i(s,\pi)>0$. If $\ell\le m$, then tag all $\ell$ of these. If $\ell>m$, then tag the $m$ that have the largest $R_i(s,\pi)$. Ties can be broken arbitrarily, for instance by simple randomization. Let $\mathcal I_s$ denote the set of all boxes that are tagged under the optimal strategy when the observed signal is $S=s$. The expected reward, given signal $s$, is
\begin{eqnarray*}
	R(s,\pi) := \sum_{i\in\mathcal I_s} R_i(s,\pi)= \sum_{i\in\mathcal I_s} \left(c_i-(c_i+d_i)\alpha_i(s,\pi)\right)
\end{eqnarray*}
and, unconditionally, it is
\begin{eqnarray}\label{eq: expected reward}
R_{\mathrm{opt}}(\pi) := \sum_{s\in\Sigma}\sum_{\gamma\in\Gamma} R(s,\pi) \pi(\gamma)p(s|\gamma) = \sum_{s\in\Sigma}\sum_{\gamma\in\Gamma} \sum_{i\in\mathcal I_s} R_i(s,\pi) \pi(\gamma)p(s|\gamma).
\end{eqnarray}
Since we can break ties in a variety of ways, there may be multiple sets $\mathcal I_s$ that can be used. However, they all lead to the same value for the expected reward. This discussion can be summarized as follows.

\begin{thrm}
The optimal  strategy for the DM to maximize the expected total reward is given by Algorithm 1.
\end{thrm}

\noindent\textbf{Algorithm 1.} Given distribution $\pi$ and signal $S=s$.\\
\textbf{Step 1.} Use \eqref{eq: alpha i} to calculate $\alpha_i(s,\pi)$ for each $i=1,2,\dots,n$.\\
\textbf{Step 2.} Use \eqref{Ri} to calculate $R_i(s,\pi)$ for each $i=1,2,\dots,n$.\\
\textbf{Step 3.}  Let $\ell$ be the number of sites $i$ with $R_i(s,\pi)>0$. If $\ell\le m$, then tag all $\ell$ of these. If $\ell>m$, then tag the $m$ that have the largest $R_i(s,\pi)$. Ties can be broken arbitrarily.\\

For computational purposes, it is often convenient to calculate the relevant quantities for all $\gamma\in \Gamma$ and all $s\in\Sigma$ simultaneously. Toward this end, we define several vectors and matrices. First, we establish our notation for working with matrices. For a matrix $A$, we write $A^\top$ to denote the transpose of $A$. We write $(a_{ij})$ to denote the matrix with $a_{ij}$ in the $i$th row and $j$th column and given a column vector $u=(u_1,\dots,u_r)$ we write $\Diag(u)$ to denote the $r\times r$ diagonal matrix with the elements of $u$ on the diagonal. All vectors should be interpreted as column vectors for the purpose of matrix calculations.

Let $\gamma{(1)},\gamma{(2)},\dots,\gamma{(M)}$ be an arbitrary ordering of the elements in $\Gamma$ and, by a slight abuse of notation, define the vector $\pi = \left(\pi\big(\gamma{(1)}\big),\pi\big(\gamma{(2)}\big),\dots, \pi\big(\gamma{(M)}\big)\right)$. Similarly, we let $s{(1)},s{(2)},\dots,s{(L)}$ be an arbitrary ordering of the elements in $\Sigma$ and define the two  vectors $P_S(\pi)$ and $P_{S}^{(-1)}(\pi)$ by: $P_S(\pi) = \left(p(s{(1)}|\pi),p(s{(2)}|\pi), \dots,p(s{(L)}|\pi)\right)$ and $P_{S}^{(-1)}(\pi) = \left(1/p(s{(1)}|\pi),1/p(s{(2)}|\pi),\dots,1/p(s{(L)}|\pi)\right)$. Next, let $G=(\gamma_j{(i)})$ be an $M\times n$ matrix, let $P=\{p(s{(j)}|\gamma{(i)})\}$ be an $M\times L$ stochastic matrix, let $\Theta(\pi)=\{\theta(\gamma{(j)}|s{(i)},\pi)\}$ be an $L\times M$ stochastic matrix, and let $A(\pi)=\{\alpha_j(s{(i)},\pi)\}$ be an $L\times n$ matrix. Proposition \ref{Prop1} implies the following matrix formulations:
\begin{eqnarray}\label{eq: matrix form}
P_S(\pi)=P^\top\pi, \quad \Theta(\pi)=\Diag(P_{S}^{(-1)})P^\top\Diag(\pi), \quad \mbox{and }\quad A(\pi)=\Theta(\pi)G^c,
\end{eqnarray}
where matrix $G^c$ is a ``complementary" matrix to $G$, i.e., all zeroes and ones in $G$ are interchanged. Next, let $c=(c_1,c_2,\dots,c_n)$ and let $d=(d_1,d_2,\dots,d_n)$ be the vectors of rewards and penalties, respectively, and let $C^*$ be the $L\times n$ matrix, where, for $j=1,2,\dots,n$, all elements of the $j$th column are given by $c_j$. In this context, Algorithm 1 becomes as follows.\\

\noindent\textbf{Algorithm 2.} Given distribution $\pi$ and signal $S=s(i)$ for a given $i$.\\
\textbf{Step 1.} Use \eqref{eq: matrix form} to calculate the matrix $A(\pi)$.\\
\textbf{Step 2.} Calculate the $L\times n$ matrix $R^*(\pi) = C^* -A(\pi)\Diag(c+d)$.\\
\textbf{Step 3.} Let $\ell$ be the number of elements in the $i$th row of $R^*(\pi)$ with positive entries.  If $\ell\le m$, then tag all $\ell$ sites corresponding to these locations. If $\ell>m$, then tag the $m$ sites that correspond to the locations of the largest $R^*_{ij}(\pi)$. Ties can be broken arbitrarily.\\

Note that the steps in Algorithm 2  can be computed for all possible signals $s\in \Sigma$ simultaneously. Thus, we can precompute the decisions ahead of time, which makes it easy to implement the procedure in high-throughput applications. However, the steps of this algorithm require working with and multiplying matrices that may be very high-dimensional, which can be computationally challenging. A computationally tractable solution can be found for the symmetric case and is described in Section \ref{sec: STF} below.

\section{Identical Testing}\label{sec: ident test}

In many practical situations the tests are identical in that the parameters of testing are the same for all boxes, i.e.,\ $a_i=a\ge1/2$, $b_i=b\ge1/2$ for each $i=1,2,\dots,n$. This happens, e.g., in a medical context where the same screening is administered to different patients. In this section we give some interesting results for this situation. They will really pay off when we discuss the symmetric case below.

Up to now we have worked with two main random variables: the signal $S$ and the allocation of objects $T$. We now introduce three more random variables. Let $N$ be the number of zeros in signal $S$, let $N_1$ be the number of these zeros that are in boxes with an object, i.e., the number of {\em false zeros}, and let $N_2$ be the number of these zeros in empty boxes, i.e., the number of {\em correct zeros}. Note that $N= N_1 + N_2$, $0\le N\le n$, $0\le N_2\le n-k$, and
\begin{eqnarray}\label{eq: bound on N1}
0\vee\left(k-n+N\right)\le N_1\le N\wedge k,
\end{eqnarray}
where we use $\wedge$ and $\vee$ to denote, respectively, the minimum and maximum of two real numbers. 

Since the DM only knows $S$ and not $T$, the DM knows the value of $N$, but not the values of $N_1$ or $N_2$. Interestingly enough, the random variables $N$, $N_1$, and $N_2$ are independent of $T$. Even more interesting is the fact that their distributions do not depend on $\pi$, the distribution of $T$. We now verify this fact. Here and throughout, we use write $\mathrm{Bin}(n,p)$ to denote a binomial distribution with parameters $n$ and $p$.

\begin{lemma}
The random variables $N_1$ and $N_2$ are independent with $N_1\sim \mathrm{Bin}(k,1-a)$ and $N_2 \sim \mathrm{Bin}(n-k,b)$. Moreover, $N_1$ and $N_2$ are independent of $T$.
\end{lemma}

\begin{proof}
Note that conditionally, given $T$, we have $N_1\sim \mathrm{Bin}(k,1-a)$ and $N_2 \sim \mathrm{Bin}(n-k,b)$. Furthermore, given $T$, $N_1$ and $N_2$ are conditionally independent since $N_1$ only depends on $S_i$ where $T_i=1$, $N_2$ only depends on $S_i$ where $T_i=0$, and the $S_i$'s are conditionally independent. Since the conditional joint distribution of $N_1$ and $N_2$ does not depend on $T$, the random variables are independent of $T$ and have the required joint distribution.
\end{proof}

Since $N=N_1+N_2$, $N$ is independent of $T$ and its distribution does not depend on $\pi$ as it is the convolution of the distributions of $N_1$ and $N_2$. Unless $b=1-a$, this distribution is not binomial. Instead, it is the so-called Poisson binomial distribution, see, e.g., \cite{Hong:2013} and the references therein. Let  $p_1$, $p_2$, and $g_{n,k}$ be the pmfs of $N_1$, $N_2$, and $N$, respectively. By the standard discrete convolution formula, we have
\begin{eqnarray}
g_{n,k}(x)=\sum_{t} p_1(t)p_2(x-t) = \sum_{t}p_1(x-t)p_2(t), \quad x =0,1,\dots, n. \label{con}
\end{eqnarray}
The following distribution also plays an important role in our analysis: the joint \dst of $N_1$ and $N$ has pmf
\begin{eqnarray}
s(t,x) &:=& P(N_1=t,N=x)= P(N_1=t,N_2=x-t)\nonumber \\
&=& p_1(t)p_2(x-t)={k\choose t}(1-a)^ta^{k-t}{n-k\choose x-t} b^{x-t}(1-b)^{n-k-x+t}
\label{stx}
\end{eqnarray}
for $0\leq x\leq n$, $0\vee(k-n+x)\le t\le x\wedge k$.

Before proceeding, we define three functions that are closely related to the random variables $N$, $N_1$, and $N_2$. Let $\mathbb N_n=\{0,1,2,\dots,n\}$, let $M:\Sigma\mapsto\mathbb N_n$, and let $M_1,M_2:\Gamma\times\Sigma\mapsto\mathbb N_n$, where $M(s)=\sum_{i=1}^n (1-s_i)$, $M_1(\gamma,s)=\sum_{i=1}^n \gamma_i (1-s_i)$, and $M_2(\gamma,s) = \sum_{i=1}^n (1-\gamma_i) (1-s_i)$.  It is readily checked that $M(s)=M_1(\gamma,s)+M_2(\gamma,s)$ and that $N=M(S)$, $N_1 = M_1(T,S)$, and $N_2=M_2(T,S)$.  With this notation, let us denote 
$$
G(t,x)=\{(\gamma,s): M_1(\gamma,s)=t, M(s)=x\}, \ 0\leq x\leq n, \ 0\vee(k-n+x)\le t\le x\wedge k, 
$$
and note that these sets form a partition of $\Omega=\Gamma\times \Sigma$. In words, $G(t,x)$ is the set of all pairs $(\gamma,s)$ such that $s$ has exactly $x$ zeros and, when $T=\gamma$, it has exactly $t$ false zeros and $x-t$ true zeros.

\begin{prop}\label{Lem1}
a) For all $\gamma \in \Gamma$ and all $s\in\Sigma$ with $(\gamma,s)\in G(t,x)$, we have
\begin{eqnarray}
p(s|\gamma) =a^{k-t}(1-a)^{t}b^{x-t}(1-b)^{n+t-k-x}=: p(t,x). \label{ptx}
\end{eqnarray}
b) For fixed $x$, the function $p(t,x)$ is monotonically decreasing in $t$. \\
c) For all $t,x$ with $0\leq x\leq n$ and $0\vee(k-n+x)\le t\le x\wedge k$, we have 
\begin{eqnarray}\label{Gtx}
|G(t,x)|= \binom{n}{k}\binom{k}{t}\binom{n-k}{x-t}=\binom{n}{x}\binom{x}{t}\binom{n-x}{k-t},
\end{eqnarray}
\begin{eqnarray}
p(t,x)|G(t,x)|=\binom{n}{k}s(t,x), \ \mbox{and} \ \ p(t,x)=\frac{s(t,x)}{\binom{k}{t}\binom{n-k}{x-t}}. \label{tstx}
\end{eqnarray}
\end{prop}

\begin{proof}
a) Note that, if $(\gamma,s)\in G(t,x)$, then the number of false zeros is $t$, the number of true zeros is $x-t$, the number of true ones is $k-t$, and the number of false ones is $n-x-k+t$. From here, the result follows from \eqref{eq: psg} and the definitions of $N_1$ and $N$.

b) The result follows from the fact that $\frac{p(t+1,x)}{p(t,x)}=\frac{(1-a)(1-b)}{ab}\le1$ for $a,b\ge1/2$. 

c) Fix $(\gamma,s)\in G(t,x)$. The number of choices for $\gamma$ is $\binom{n}{k}$ as there is no restriction on $\gamma$. For a given $\gamma$ we can only have $s$ that has exactly $x$ zeros with $t$ of them in locations for which $\gamma$ has a $1$ and $x-t$ of them in locations where $\gamma$ has a $0$. The number of ways to select the former is $\binom{k}{t}$ and the number of ways to select the latter is $\binom{n-k}{x-t}$. Now, applying the Product Principle gives the first equality in \eqref{Gtx}. The second equality in \eqref{Gtx} follows by writing both sides using factorials and simplifying. The second equality in \eqref{tstx} follows from \eqref{stx}. The first equality in \eqref{tstx} follows from the second by applying \eqref{Gtx}.
\end{proof}

\section{Illustrative Example} \label{sec: Example}

In this section we apply Algorithm 1 to fully solve the TF problem with $n=2$ boxes and $k=1$ objects. For the number of tags, we consider both possible cases: $m=1,2$.  For simplicity, we focus on the Identical Testing case, where the parameters of testing are the same for the two boxes with $a_i=a\ge1/2$, $b_i=b\ge1/2$. However, we allow the cost parameters, $c_1,d_1, c_2,d_2$, to be distinct and we make no assumptions on the distribution $\pi$. We will see that, even in this simple situation, a variety of interesting scenarios can arise. 

There are $\binom{2}{1}=2$ possible positions of the object: $\gamma(1)=(1,0)$ and $\gamma(2)=(0,1)$. Thus $\Gamma=\left\{\gamma(1),\gamma(2)\right\}$ and our prior distribution can be written as $\pi=( \pi(\gamma(1)),\pi(\gamma(2)) )=(x,1-x)$ for some $x\in[0,1]$. Here, $x$ is the probability that the object is in Box $1$ and $1-x$ is the probability that it is in Box $2$. Since $x$ uniquely determines $\pi$, we will replace $\pi$ with $x$ in our notation, e.g., we write $p(s|x)$ instead of $p(s|\pi)$. 

There are $2^2=4$ possible signals: $s(1)=(1,1)$, $s(2)=(1,0)$, $s(3)=(0,1)$, $s(4)=(0,0)$ and thus $\Sigma=\{s(j):j=1,2,3,4\}$. Using \eqref{eq: psg} we can obtain all $8$ values of  $p(s|\gamma)$, $s\in\Sigma,\ \gamma\in\Gamma$. For $j=1,2,3,4$, we have $p(s(j)|\gamma(1))=e_j$, where $e_1=a(1-b)$, $e_2=ab$, $e_3=(1-a)(1-b)$, and $e_4=(1-a)b$. Similarly, the values of $p(s(j)|\gamma(2))$ are $e_1, e_3, e_2, e_4$ for $j=1,2,3,4$, respectively. This is summarized in Columns 2 and 3 of Table \ref{Tab:STF}. The probabilities of $p(s|x)$ can be obtained using \eqref{pspi} and are as follows: $p(s(1)|x)=e_1$, $p(s(2)|x)= xe_2+(1-x)e_3=: q_2(x)$, $p(s(3)|x)=xe_3+(1-x)e_2=: q_3(x)$, $p(s(4)|x)=e_4$. These are summarized in Column 4 of Table  \ref{Tab:STF}. Next, we calculate values for the $\alpha_i$'s using \eqref{eq: alpha i}. The results are given in Columns 5 and 6 of Table~\ref{Tab:STF}. Applying \eqref{Ri} gives 
\begin{eqnarray*}
R_1(s(1),x)&=&R_1(s(4),x)=c_1-(c_1+d_1)(1-x), \\
R_1(s(2),x)&=&c_1-\frac{(c_1+d_1)(1-x)e_3}{xe_2+(1-x)e_3}, \\
R_1(s(3),x)&=&c_1-\frac{(c_1+d_1)(1-x)e_2}{xe_3+(1-x)e_2},\\
 R_2(s(1),x) &=&R_2(s(4),x)= c_2-(c_2+d_2)x,\\
 R_2(s(2),x)&=& c_2- \frac{(c_2+d_2)xe_2}{xe_2+(1-x)e_3},\mbox{ and }\\
 R_2(s(3),x)&=& c_2-\frac{(c_2+d_2)xe_3}{xe_3+(1-x)e_2}.
\end{eqnarray*}
Note that, when $a=b=1/2$, which corresponds to the non-informative case, the value of $R_i(s(j),x)$ does not depend on $j$. If, in addition, $x=1/2$, then $R_i(s(j),1/2) = 0.5(c_i-d_i)$ for each $i,j$. 

\begin{table}
\begin{center}
	\begin{tabular}
{|c|c|c|c|c|c|} 
         \hline   $ $  & $p(s|\gamma(1))$ & $p(s|\gamma(2))$ & $p(s|\pi)$ & $\alpha_1(s,\pi)$ 
         & $\alpha_2(s,\pi)$ 
    \\
		\hline
         $s(1)=(1,1) $  & $e_1$  & $e_1$   &$e_1$      &  $1-x $  & $x$    \\
		\hline
         $ s(2)=(1,0)$   & $e_2$  & $e_3$    & $q_2(x)$  & $(1-x)e_3/q_2(x)$  & $xe_2/q_2(x)$  \\
		\hline
		$s(3)=(0,1)$     & $e_3$ & $e_2$     & $q_3(x)$  & $(1-x)e_2/q_3(x)$ & $xe_3/q_3(x)$   \\
        \hline
		$s(4)=(0,0)$    & $e_4$ & $e_4 $    &  $e_4$        & $1-x $    & $x$   \\
         \hline
	\end{tabular}
\end{center}
\caption{Summary of relevant quantities for the TF problem with parameters: $n=2$, $k=1$, and probability $x$ that the object is in Box $1$. Quantities $e_1$--$e_4$ and $q_2(x),q_3(x)$ are defined in the text.}\label{Tab:STF}
\end{table}

It is easily checked that $\frac{\partial}{\partial x} \alpha_1(s,x)<0$ for each $s\in\Sigma$ and that, for fixed $s$, the values of $\alpha_1(s,x)$ are monotonically decreasing from $1$ to $0$ as $x$ increases from $0$ to $1$. For $\alpha_2$ the opposite is true, which can be seen from the fact that \eqref{snk} implies that $\alpha_1(s,x)+\alpha_2(s,x)=n-k=1$.  From here, \eqref{Ri} implies that $R_1(s,x)$ is monotonically increasing from negative values to positive, while $R_2(s,x)$ does the reverse. Let $x_i(j)$ be the root of $R_i(s(j),x)$ for $i=1,2$ and $j=1,2,3,4$. It is readily checked that $x_1(1) = x_1(4) = \frac{d_1}{c_1+d_1}$, $x_2(1) = x_2(4) = \frac{c_2}{c_2+d_2}$, $x_1(2) = \frac{d_1 e_3}{c_1 e_2+d_1e_3}$, $x_2(2) = \frac{c_2 e_3}{c_2 e_3+d_2e_2}$, $x_1(3) = \frac{d_1 e_2}{c_1 e_3+d_1e_2}$, and $x_2(3) = \frac{c_2 e_2}{c_2 e_2+d_2e_3}$. This means that for each $j$, the interval $[0,1]$ contains two subintervals $I_1(j)=\left(x_1(j),1\right]$ and $I_2(j)=\left[0,x_2(j)\right)$ such that, if signal $s(j)$ is observed, then on $I_1(j)$ we expect to get a positive reward if we tag Box 1 and the same for Box 2 on  $I_2(j)$. For some values of the parameters these intervals may overlap and for others not. 

Assume, for the moment, that signal $s(j)$ is observed and that the intervals overlap, i.e.,\ that $x_2(j)> x_1(j)$. In this case, if the number of tags is $m=2$, then DM will tag only Box 1 if $1\ge x> x_2(j)$,  tag both boxes when $x_2(j)>x>x_1(j)$, and tag only Box 2 when $0\le x\le x_1(j)$. When $m=1$, we need another point, $x_*(j)$, which is the unique point satisfying $R_1(s(j),x_*(j))=R_2(s(j),x_*(j))$. We necessarily have $x_1(j)\le x_*(j)\le x_2(j)$. In this case, DM will tag Box $1$ when $x_*(j)<x\le1$ and DM will tag Box $2$ when $0\le x<x_*(j)$. Note that, when $1\ge x> x_2(j)$ and when $0\le x\le x_1(j)$, the DM makes the same decision no matter the value of $m$. 

Now assume that  signal $s(j)$ is observed and that $x_2(j)< x_1(j)$. In this case DM will tag only Box 1 if $1\ge x> x_2(j)$, not tag any box when $x_2(j)>x>x_1(j)$, and tag only Box 2 when $0\le x\le x_1(j)$. Here, the decision is the same regardless of $m$. 

To better illustrate our results, in Figures  \ref{Fig:STF_case1} and \ref{Fig:STF_case2} we plot $R_1$ and $R_2$ for several choices of the parameters. In Figure \ref{Fig:STF_case1} we consider the case with  $a=2/3$, $b=3/4$, $c_1=2$, $c_2 =3$, $d_1=1$, and $d_2=2$. Here, we have $x_2(j)> x_1(j)$ for all $j$ and thus, we would follow different strategies depending on the value of $m$. In Figure \ref{Fig:STF_case2} we consider a situation with mostly the same parameters except that now $d_1=2$ and $d_2=5$. In this case, we have  $x_2(j)<x_1(j)$ for all $j$ and we would never tag more than one box at a time.

\begin{figure}[tp!]
\includegraphics[width=0.9\textwidth, height = 5.5 cm ]{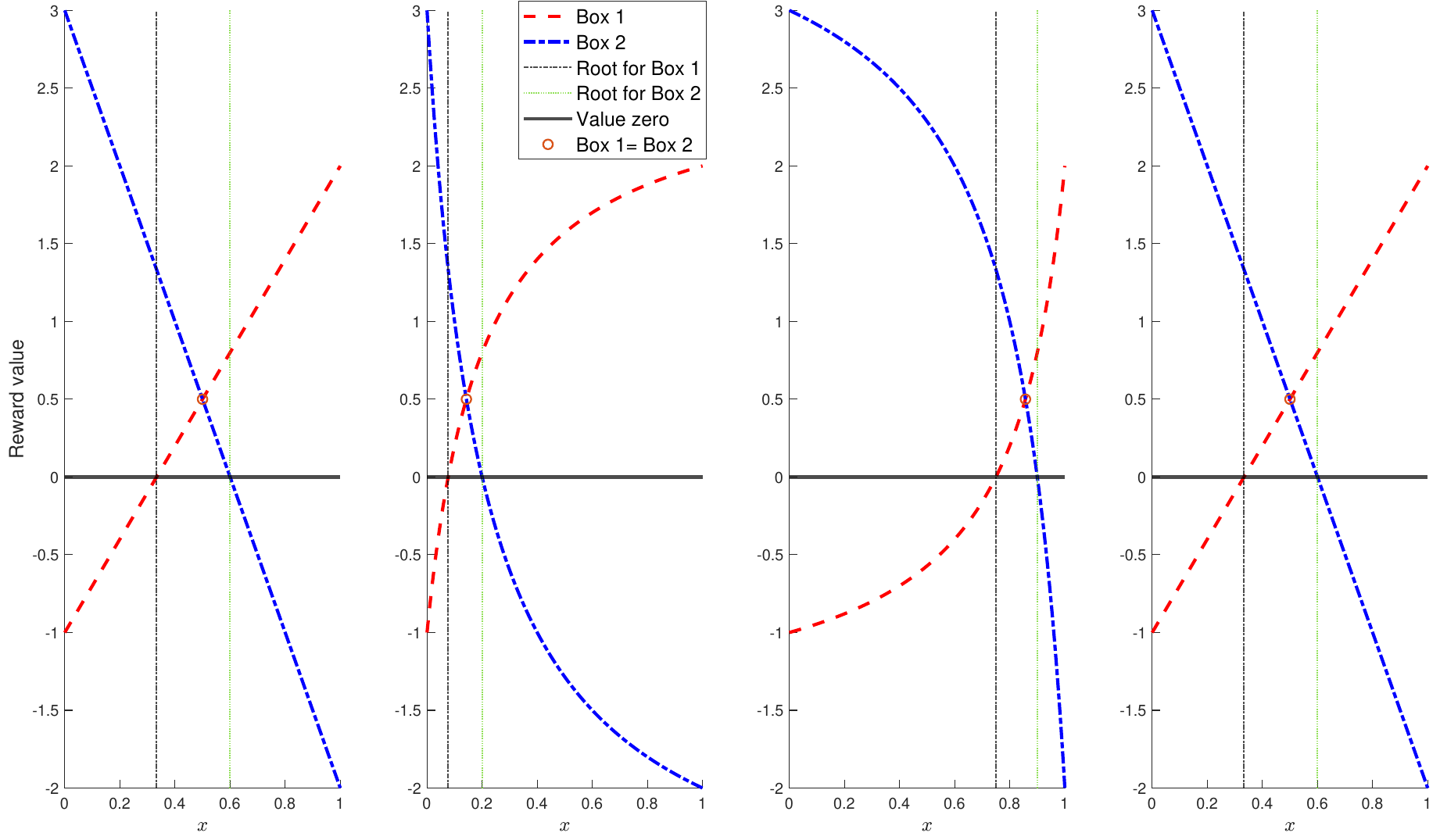}
\caption{Plots for the TF problem with parameters: $n=2$, $k=1$, $a=2/3$, $b=3/4$, $c_1=2$, $c_2 =3$, $d_1=1$, and $d_2=2$. The four subplots plot the expected reward values $R_1(s(j))$ and $R_2(s(j))$ for $j=1,2,3,4$, respectively, against the probability $x$. The vertical lines indicate the locations where the reward values hit zeros. }\label{Fig:STF_case1}
\end{figure}

\begin{figure}[tp!]
\includegraphics[width=0.9\textwidth, height = 5.5 cm ]{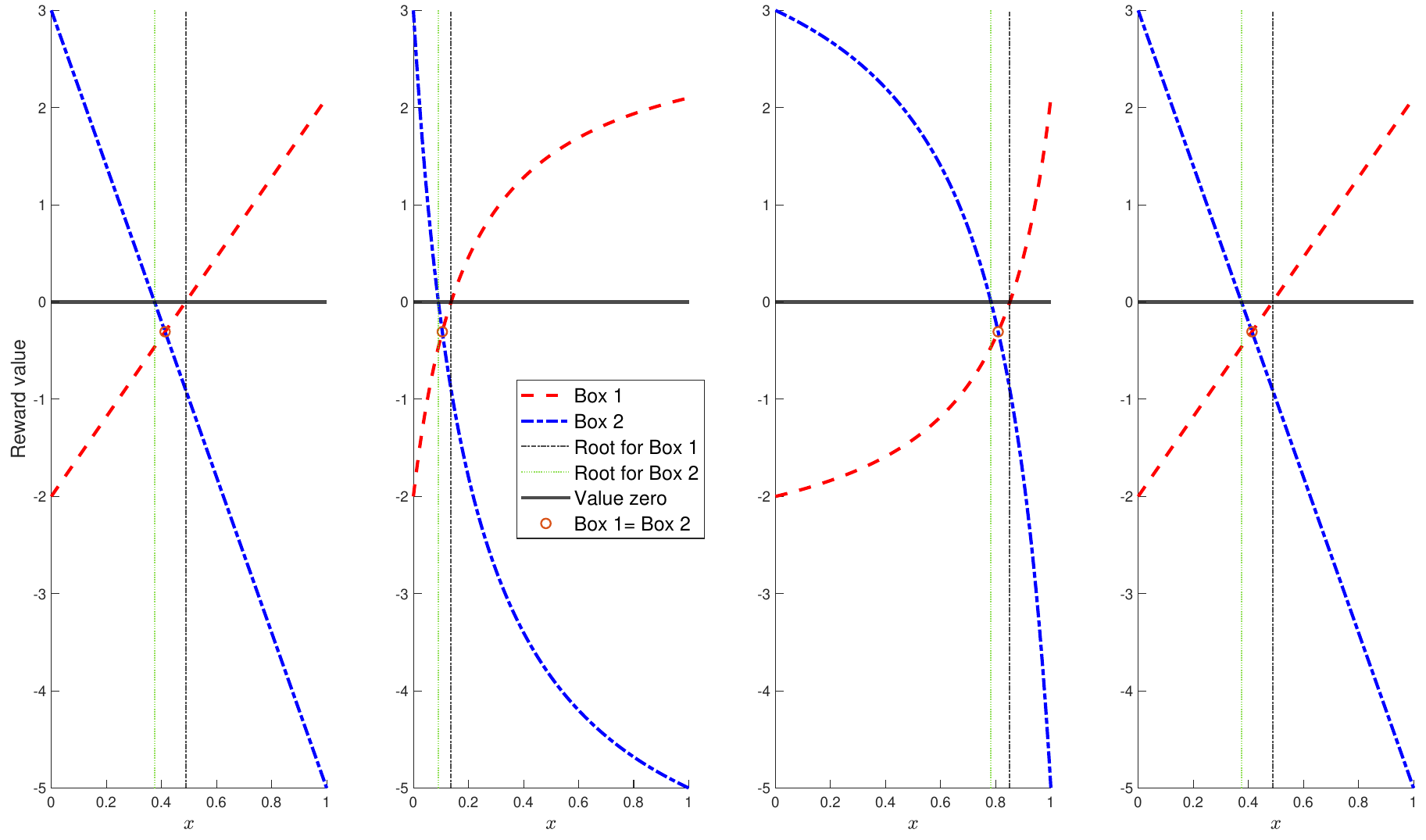}
\caption{Plots for the TF problem with parameters: $n=2$, $k=1$, $a=2/3$, $b=3/4$, $c_1=2$, $c_2 =3$, $d_1=2$, and $d_2=5$. The four subplots plot the expected reward values $R_1(s(j))$ and $R_2(s(j))$ for $j=1,2,3,4$, respectively, against the probability $x$. The vertical lines indicate the locations where the reward values hit zeros. 
}\label{Fig:STF_case2}
\end{figure}

\section{Symmetric Test and Find Model} \label{sec: STF}

While the solution to the general TF problem is straightforward, it requires dealing with $M={n \choose k}$ allocations of the objects and $L=2^n$ signals, which grows exponentially fast.  In this section we consider a symmetric situation, where one can get much more explicit results and the size of the problem grows only polynomially. As in Section \ref{sec: ident test}, we assume that we have identical testing, i.e.,\ $a_i=a$ and $b_i=b$ for all $i$, and, throughout, we use the notation defined in that section. In addition, we now assume that $\pi=\pi_*$ is the uniform \dst on $\Gamma$. We still allow the cost structure to be different in different boxes and, thus, the values of $c_i$ and $d_i$ may still depend on $i$. We refer to this situation as the symmetric case. Since everything in this section is done under the uniform distribution $\pi_*$, we do not emphasize it in the notation as we had in the previous sections. Here, any combination of $k$ filled boxes, i.e.,\ every $\gamma\in\Gamma$, has the same probability $\pi_*(\gamma) = P(T=\gamma)=1/\binom{n}{k}$. It follows that
\begin{eqnarray}\label{eq: uncond prob Ti}
P(T_i=1)=\frac{\binom{n-1}{k-1}}{\binom{n}{k}}=\frac{k}{n} \mbox{ and }  P(T_i=0)=\frac{n-k}{n}.
\end{eqnarray}
The solution in this case is influenced by the approach to solving the symmetric LBT problem in \cite{son25}. We now give several results for this case.

\begin{prop}\label{Lem1b}
a) Under $\pi_*$, for any $s\in\Sigma$ and any $x=0,1,\dots,n$, when $M(s)=x$, we have
\begin{eqnarray}
p(s) = \frac{g_{n,k}(x)}{\binom{n}{x}} \mbox{ and } P(S=s|N=x)=\frac{1}{\binom{n}{x}}.
 \label{psx}
\end{eqnarray}
b) Under $\pi_*$, for all $\gamma \in \Gamma$ and all $s\in\Sigma$ with $(\gamma,s)\in G(t,x)$, we have
$$
\theta(\gamma|s)= \frac{{n\choose x} p(t,x)}{{n\choose k} g_{n,k}(x)} =\frac{\binom{n}{x}s(t,x)}{|G(t,x)|g_{n,k}(x)}=: \theta(t,x).
$$
c) For fixed $x$, the function $\theta(t,x)$ is monotonically decreasing in $t$. \\
d) We have 
\begin{eqnarray*} 
 \sum_{t,x} p(t,x)|G(t,x)|=\binom{n}{k} , \ \ \sum_{t,x} \theta(t,x)|G(t,x)|=2^n, 
 \ \ \mbox{and} \ \ \sum_{t,x}|G(t,x)|=|\Omega|=\binom{n}{k}2^n.
\end{eqnarray*}
\end{prop}

The second equality in \eqref{psx} is the intuitively appealing observation that, under $\pi_*$, the conditional distribution of $S$, given the number of zeros, is uniform. The formulas in Part d) are useful for checking calculations. The fact that $\theta(t,x)$ is monotonically decreasing in $t$ has the following intuitive interpretation. While the prior distribution of $\gamma$ is uniform, after observing $S=s$, the posterior distribution assigns greater probability to values of $\gamma$ that result in fewer errors, i.e., those with smaller values of $M_1(\gamma,s)$.

\begin{proof}
a) We have
\begin{eqnarray*}
p(s) = \sum_{\gamma\in\Gamma}p(s|\gamma)\pi_*(\gamma) = \frac{1}{{n\choose k}}\sum_{i=0}^x p(i,x)\frac{|G(i,x)|}{{n\choose x}} = \frac{1}{{n\choose x}}\sum_{i=0}^x s(i,x) = \frac{P(N=x)}{{n\choose x}},
\end{eqnarray*}
where the second equality follows from the fact that $|G(i,x)|={n\choose x}|\{\gamma\in\Gamma:(\gamma,s)\in G(i,x)\}|$, the third equality follows from \eqref{tstx}, and the last equality from the definition of $s(i,x)$ in \eqref{stx}. From here, the first equality in \eqref{psx} follows from the definition of $g_{n,k}$ and the second from the fact that, when $M(s)=x$, we have $P(S=s) = P(S=s, N=x)$. 

b) The result follows by applying \eqref{pgs}, \eqref{psx}, and \eqref{tstx}.

c) This follows by combining the corresponding result for $p(t,x)$ given in Proposition \ref{Lem1} with Part b).

d) The first equality follows from \eqref{tstx} and the fact that $s(t,x)$ is a joint pmf. To see that the second equality holds note that Part b) implies that $\theta(t,x)|G(t,x)|=\frac{{n\choose x}s(t,x)}{g_{n,k}(x)}$. From here \eqref{con} combined with \eqref{stx} gives
$$
\sum_{x,t}\theta(t,x)|G(t,x)| = \sum_{x=0}^n \frac{{n\choose x}}{g_{n,k}(x)} \sum_{t}s(t,x) = \sum_x {n\choose x}=2^n.
$$
The third equality follows from the fact that the sets $G(t,x)$ form a partition of $\Omega$.
\end{proof}

We now give explicit formulas for $\alpha_i(s)=P(T_i=0|S=s)$ in the symmetric TF problem.

\begin{prop}\label{Lem2}
Fix $i\in\{1,2,\dots,n\}$, $s\in \Sigma$, and let $x=M(s)$. If $s_i=0$, then $x\in\{1,2,\dots,n\}$ and
$$
\alpha_i(s) =b \frac{(n-k)}{x}\frac{g_{n-1,k}(x-1)}{g_{n,k}(x)}=:\alpha^-(x,k).
$$
If $s_i=1$, then $x\in\{0,1,\dots,n-1\}$ and
$$
\alpha_i(s) = (1-b)\frac{(n-k)}{n-x}\frac{g_{n-1,k}(x)}{g_{n,k}(x)}=: \alpha^+(x,k).
$$
Moreover, we have
$$
x  \alpha^-(x,k) + (n-x) \alpha^+(x,k) = n-k.
$$
\end{prop}

Note that the formula for $\alpha_i(s)$ depends on $s$ only through $s_i$ and $x$. Thus,
\begin{equation*}
\alpha_i(s)=P(T_{i}=0|S=s)=P(T_{i}=0|S_i=s_{i},N=x).
\label{alx}
\end{equation*}
Some properties of functions of the form $\alpha^\pm(x,k)$ are discussed in \cite{son25}.

\begin{proof}
First note that, for $s\in\Sigma$, it is impossible to have $s_i=0$ and $M(s)=0$ or $s_i=1$ and $M(s)=n$. Next, let $S_{-i}$ and $s_{-1}$ denote, respectively, the vectors $S$ and $s$ without coordinate $i$. Applying Bayes' rule gives
\begin{eqnarray*}
\alpha_i(s) &=& P(T_i=0|S=s_i)= \frac{ P(S=s| T_i=0)P(T_i=0) }{P(S=s)} \\
&=& \frac{ P(S_{-i}=s_{-i}, S_i=s_i| T_i=0)P(T_i=0) }{P(S=s)} \\
&=& \frac{P(S_{-i}=s_{-i}| T_i=0) P(S_i=s_i| T_i=0)P(T_i=0) }{P(S=s)} \\
&=& \frac{n-k}{n}\binom{n}{x} \frac{P(S_i=s_i|T_i=0)P(S_{-i}=s_{-i}|T_i=0)}{g_{n,k}(x)},
\end{eqnarray*}
where the third line follows by Lemma \ref{lemma: cond indep} and the fourth follows by \eqref{eq: uncond prob Ti} and \eqref{psx}.

When $s_i=0$ we have $P(S_i=0|T_i=0)=b$ and $P(S_{-i}=s_{-i}|T_i=0)=g_{n-1,k}(x-1)/\binom{n-1}{x-1}$. To see that the latter holds, note that, in this case, $s_{-i}$ has $x-1$ zeros and thus we just need to find $p(s)$ in a situation with $n-1$ boxes, $x-1$ zeros, and $k$ objects, which we can do by applying \eqref{psx}. From here $\alpha_i(s) = \alpha^-(x,k)$ follows from the easily verified  the fact that $\frac{n-k}{n}{n\choose x}/{n-1\choose x-1}=\frac{n-k}{x}$.

When $s_i=1$ we have $P(S_i=1|T_i=0)=1-b$ and $P(S_{-i}=s_{-i}|T_i=0)=g_{n-1,k}(x)/\binom{n-1}{x}$, where the latter holds by applying \eqref{psx} to a situation with $n-1$ boxes, $x$ zeros, and $k$ objects. From here $\alpha_i(s) = \alpha^+(x,k)$ follows from the easily verified fact that $\frac{n-k}{n}{n\choose x}/{n-1\choose x}=\frac{n-k}{n-x}$.

The last equality, follows from \eqref{snk}, the fact that $\alpha_i$ depends on $i$ only through $s_i$, and the fact that $x$ of the $s_i$ equal $0$ and the remaining $n-x$ equal $1$. 
\end{proof}

We now summarize the solution in the symmetric case.

\begin{thrm}
In the symmetric case, the optimal  strategy for the DM to maximize the expected total reward is given by Algorithm 3.
\end{thrm}

\noindent\textbf{Algorithm 3.} Given the uniform distribution $\pi=\pi^*$ and signal $S=s$.\\
\textbf{Step 1.} Use Proposition \ref{Lem2} to calculate $\alpha_i(s)$ for each $i=1,2,\dots,n$.\\
\textbf{Step 2.} Use \eqref{Ri} to calculate $R_i(s)$ for each $i=1,2,\dots,n$.\\
\textbf{Step 3.}  Let $\ell$ be the number of sites $i$ with $R_i(s)>0$. If $\ell\le m$, then tag all $\ell$ of these. If $\ell>m$, then tag the $m$ that have the largest $R_i(s)$. Ties can be broken arbitrarily.\\

We now give the expected reward in the symmetric case. Let $\mathcal I_s$ denote the set of all boxes that are tagged when following the optimal policy in the event that signal $S=s$ is observed. The expected reward from following the optimal policy is
\begin{eqnarray*}
R_{\mathrm{opt}}(\pi_*)  = \sum_{s\in\Sigma} \sum_{i\in\mathcal I_s} R_i(s) p(s) = \sum_{s\in\Sigma}\sum_{i\in\mathcal I_s} R_i(s)  \frac{g_{n,k}(M(s))}{{n\choose M(s)}}.
\end{eqnarray*}

\section{Extensions}\label{sec: Extensions}

In this section we consider several extensions of the TF problem.

\subsection{Extension I: Identify All Boxes}\label{sec: gen variant}

In this section we consider the problem, where we only get a reward if we correctly identify {\em all} boxes that contain objects and all that do not.  If even one box is not correctly identified, then we get nothing. In this case, we must find the $\gamma^*\in\Gamma$ that has the highest posterior probability $\theta(\gamma|s,\pi)$. We then tag all boxes that correspond to a $1$ in $\gamma^*$. Proposition \ref{Prop1} implies that
$$
\gamma^* = \underset{\gamma\in\Gamma}{\operatorname{argmax}}\ \pi(\gamma)p(s|\gamma) =\underset{\gamma\in\Gamma}{\operatorname{argmax}}\ \pi(\gamma)\prod_{i: \gamma_i=1}a_i^{s_i}(1-a_i)^{1-s_i}\prod_{i: \gamma_i=0}b_i^{1-s_i}(1-b_i)^{s_i}.
$$
When more than one $\gamma\in\Gamma$ maximizes the equation, we choose any one of them arbitrarily. Assume that we receive reward $C>0$ for correctly identifying the vector and $0$ otherwise. Conditionally, given signal $s$, our expected reward is $C \theta(\gamma^*|s,\pi)$ and, unconditionally, it is
$$
C \sum_{s\in\Sigma} p(s|\pi)\max_{\gamma\in\Gamma} \pi(\gamma)p(s|\gamma) = C \sum_{s\in\Sigma} \max_{\gamma\in\Gamma}\left( \pi(\gamma)\prod_{i: \gamma_i=1}a_i^{s_i}(1-a_i)^{1-s_i}\prod_{i: \gamma_i=0}b_i^{1-s_i}(1-b_i)^{s_i} \right).
$$

We now turn to the symmetric case, where the prior distribution is uniform, i.e., $\pi=\pi_*$, and all of the tests have the same parameters. In this case, when we observe $S=s$ with $M(s) =x$, the formula for $p(s|\gamma)=p(t,x)$ is as in \eqref{ptx}. Using this simplification and the fact that $\pi_*(\gamma)$ does not depend on $\gamma$, it follows that we should select any $\gamma^*$ such that $M_1(\gamma^*,s)=t^*$ with
$$
t^* = \underset{t\in\{(k-n+x)\vee0,\dots,x\wedge k\}}{\operatorname{argmax}}\ a^{k-t}(1-a)^{t}b^{x-t}(1-b)^{n+t-k-x}
$$
or equivalently
$$
t^* = \underset{t\in\{(k-n+x)\vee0,\dots,x\wedge k\}}{\operatorname{argmax}}\ \left(\frac{(1-a)(1-b)}{ab}\right)^{t},
$$
where the possible values of $t$ are determined by \eqref{eq: bound on N1}. So long as $1/2\le a,b<1$ with at least one of $a\ne1/2$ or $b\ne1/2$, the term in parentheses will belong to $(0,1)$ and hence we take $t^*=(k-n+x)\vee0$. In the case when $x\le n-k$, we take any $\gamma^*\in\Gamma$ that leads to no false zeros and when more than one $\gamma\in\Gamma$ satisfies this, we select any one of them arbitrarily. When $x>n-k$, false zeros are unavoidable, and we must select any $\gamma\in\Gamma$ that yields exactly $x-n+k$ false zeros. Note that, in the symmetric case, we only deal with $s$ and $\gamma$ through the quantities $x=M(s)$ and $t=M_1(\gamma,s)$.

\subsection{Extension II: Identify Exactly $m$ Boxes}\label{sec: gen variant gen}

In this section we consider the  problem, where we have $m\le k$ tags and we only get a reward if we correctly identify exactly $m$ objects. When $m=k$, this reduces to the problem studied in Section \ref{sec: gen variant}. 

Let $\mathcal H_m$ be the collection of all subsets of $\{1,2,\dots,n\}$ that have exactly $m$ elements. For each $\eta\in\mathcal H_m$, let
$$
\beta_\eta(s,\pi) = P\left(\bigcap_{i\in\eta}\left[T_i=1\right]\ \middle|\ S=s\right)
$$
be the conditional probability, given that we observed signal $S=s$, that there are objects in each location in $\eta$. By arguments similar to those in the proof of \eqref{eq: alpha i}, we have
$$
\beta_\eta(s,\pi) = \frac{\sum\limits_{\gamma\in\Gamma: \gamma_i=1\ \forall i\in\eta} \pi(\gamma)p(s|\gamma)}{\sum_{\gamma\in \Gamma}\pi(\gamma)p(s|\gamma)}.
$$
It follows that we should tag the boxes whose indices are in an $\eta$ that maximizes this quantity. Equivalently, we select the indices in
\begin{eqnarray}\label{eq:eta star}
\eta^*(s) = \underset{\eta\in\mathcal H_m}{\operatorname{argmax}}\  \sum\limits_{\gamma\in\Gamma: \gamma_i=1\ \forall i\in\eta} \pi(\gamma)p(s|\gamma).
\end{eqnarray}
Ties are broken arbitrarily. Note that, for $m=k$, this is exactly the strategy described in Section \ref{sec: gen variant}. Now, assume that we receive reward $C>0$ for correctly identifying $m$ objects and $0$ otherwise. Conditionally, given signal $s$, the expected reward of our strategy is $C \max_{\eta\in\mathcal H_m}  \beta_\eta(s,\pi)$ and, unconditionally, the expected reward is
$$
C \sum_{s\in\Sigma} p(s|\pi) \max_{\eta\in\mathcal H_m} \beta_\eta(s,\pi)  = C \sum_{s\in\Sigma}  \max_{\eta\in\mathcal H_m} \left( \sum\limits_{\gamma\in\Gamma: \gamma_i=1\ \forall i\in\eta} \pi(\gamma)p(s|\gamma) \right).
$$

One might think that a simple solution to this problem is to select the $m$ boxes with the lowest $\alpha_i$'s. However, this does not work as there may be two locations that both have small $\alpha$'s, but where the combination of the two is impossible. Consider the following example with $n=5$ locations and $k=2$ items. Assume that the tests are identical at all sites, i.e., that all of the $a_i$'s are the same and equal to some value $a$ and that all of the $b_i$'s are the same and equal to some value $b$. 
Let $\gamma(1)=(1,0,1,0,0)$,  $\gamma(2)=(0,1,0,1,0)$,  $\gamma(3)=(1,0, 0,0,1)$, $\gamma(4)=(0,1,0,0,1)$ and $\pi=(10,8,4,3,0,\dots,0)/25$. Assume that we observe $s=(1,1,1,1,1)$. Since the tests are identical, we have $p(s|\gamma(i))=a^2(1-b)^3$ for each $i=1,2,3,4$. From here we can calculate the vector of $\alpha_i$'s to be $\alpha= (11,14,15,17,18)/25$. While the lowest $\alpha$'s are for sites $1$ and $2$, we should not select $\gamma=(1,1,0,0,0)$ as it has probability $0$. 

We now turn to the symmetric case, where the prior distribution is uniform, i.e., $\pi=\pi_*$, and all of the tests have identical parameters.  In this case, \eqref{eq:eta star} is equivalent to
$$
\eta^*(s) = \underset{\eta\in\mathcal H_m}{\operatorname{argmax}}\  \sum\limits_{\gamma\in\Gamma: \gamma_i=1\ \forall i\in\eta} \left(\frac{(1-a)(1-b)}{ab}\right)^{M_1(\gamma,s)}.
$$
So long as $1/2\le a,b<1$ with at least one of $a\ne1/2$ or $b\ne1/2$, the term in parentheses belongs to $(0,1)$ and thus the maximum is attained when $M_1(\gamma,s)$ is minimized for all $\gamma$ in the appropriate set. It is easy to see that this leads to the following approach. Let $x=M(s)$ and note that $s$ has $n-x$ ones. If $n-x\ge m$ select any $m$ locations that have a $1$. If $n-x<m$, then select all locations that have a $1$ and then select any $m-n+x$ locations that have a $0$. As usual, ties are broken arbitrarily.

\subsection{Extension III: Random Number of Objects}\label{sec: gen variant rand k}

In this section we consider the case where the number of objects is random. Let $K$ be the random number of objects, let $\pi_K'$ be the distribution of $K$, and note that the support of $K$ is contained in $\{0,1,\dots,n\}$, where $n$ is the number of boxes. Objects are allocated to boxes using the following two-step procedure. First, we select $K$ using $\pi'_K$, and then, given $K=k$, we select an allocation $\gamma\in\Gamma_{n,k}$ of $k$ objects into $n$ boxes following some conditional distribution $\pi'(\cdot|k)$. Here, we take $\Gamma_{n,0}=\{(0,\dots,0)\}$ to be the set of all (there is only one) allocations of $0$ objects into $n$ boxes. 

Let $\Gamma'$ denote the set of all allocations of the objects that are possible from the two-step procedure described above and note that $\Gamma'$ consists of all vectors of the form $\gamma=(\gamma_{1},\gamma_{2},\dots,\gamma_{n})$ with $\gamma_j\in\{0,1\}$. In this case, we have $\Gamma'=\bigcup_{k=0}^n\Gamma_{n,k}=\Sigma$ and $M':=|\Gamma'|=L=2^n$. Let $\pi'$ be the distribution on $\Gamma'$ that results from the two-step procedure and note that, for $\gamma\in\Gamma'$, we have $\pi'(\gamma) = \pi'(\gamma|k)\pi'_K(k)$, where $k=\sum_{i=1}^n \gamma_i$ is the number of $1$'s in $\gamma$. All of the results in Section \ref{sec: GTF} still hold so long as we replace  $\Gamma = \Gamma_{n,k}$ with $\Gamma =\Gamma'$ and we replace $M$ with $M'$. The one exception is that \eqref{snk} becomes
$$
\sum_{i=1}^n\alpha_i(s,\pi') = n-\mathrm E[K|\pi'].
$$ 
In particular, this means that we can still use Algorithms 1 and 2 (with appropriate modifications) to find the optimal strategy.

We now turn to the case where $\pi'(\cdot|k)$ is the uniform distribution on $\Gamma_{n,k}$, for each $k=0,1,\dots,n$, while we allow $\pi'_K$ to be arbitrary. We further assume that  all of the tests have the same parameters. In this case, given $S=s$ the posterior distributions of $K$ is as follows. For any $s\in \Sigma$ with $M(s)=x$, Bayes' rule combined with \eqref{psx} gives
\begin{eqnarray*}
\pi_K'(k|s) :=P(K=k|S=s) = \frac{P(S=s|K=k)P(K=k)}{\sum_{j=0}^n P(S=s|K=j)P(K=j)}= \frac{g_{n,k}(x)\pi'_K(k)}{\sum_{j=0}^n g_{n,j}(x)\pi'_K(j)}.
\end{eqnarray*}

\begin{prop}
Fix $i\in\{1,2,\dots,n\}$, $s\in \Sigma$, and let $x=M(s)$. If $s_i=0$, then $x\in\{1,2,\dots,n\}$ and
$$
\alpha_i(s) = \mathrm E[\alpha^-(x,K)|S=s] = \sum_{k=0}^n \alpha^-(x,k)\pi_K'(k|s).
$$
If $s_i=1$, then $x\in\{0,1,\dots,n-1\}$ and
$$
\alpha_i(s) = \mathrm E[\alpha^+(x,K)|S=s] = \sum_{k=0}^n \alpha^+(x,k)\pi_K'(k|s).
$$
\end{prop}

\begin{proof}
For the first part, we have
\begin{eqnarray*}
\alpha_i(s) &=& P(T_i=0|S=s) =\sum_{k=0}^n  P(T_i=0| K=k,S=s) P(K=k|S=s) \\
&=&  \sum_{k=0}^n \alpha^-(x,k)\pi'(k|s),
\end{eqnarray*}
where the last equality follows by Proposition \ref{Lem2}. The second part can be shown in a similar manner.
\end{proof}

It may be interesting to note that, when $\pi'_K$ is the uniform distribution on $\{0,1,2,\dots,n\}$ and $\pi'(\cdot|k)$ is uniform on $\Gamma_{n,k}$ for each $k$, the distribution $\pi'$ is {\em not} uniform on $\Gamma'$ since the cardinality of $\Gamma_{n,k}$ depends on $k$. To get the uniform distribution $\pi'_*$ on $\Gamma'$, we must take
$$
\pi'_K(k) = \frac{{n\choose k}}{2^n}, \ \ \ k=0,1,2,\dots,n
$$
and then let $\pi'(\cdot|k)$ be uniform for each $k$. 

When $\pi'=\pi'_*$, the situation becomes quite simple. In this case, we place equal probability on each of the $2^n$ possible arrangements of the objects into $n$ boxes. Thus, for each $\gamma\in\Gamma'$ we have $\pi'_*(\gamma)=2^{-n}$. It is readily checked that, in this case, $T_1,T_2,\dots,T_n$ are iid random variables with
$$
P(T_i=0) = 1-P(T_i=1) = \frac{1}{2}.
$$
From here it can be shown that the $S_i$'s are independent as well. By the law of total probability, we have
$$
P(S_i=s_i) = \frac{1}{2}P(S_i=s_i|T_i=0)+\frac{1}{2}P(S_i=s_i|T_i=1).
$$
When $s_i=0$ this reduces to $(b+1-a)/2$ and when $s_i=1$ it reduces to $(1-b+a)/2$. By Bayes' theorem it follows that for any $s\in\Sigma$
\begin{eqnarray*}
\alpha_i(s) = P(T_i=0|S_i=s_i) = \frac{P(S_i=s_i|T_i=0)}{P(S_i=s_i|T_i=0)+P(S_i=s_i|T_i=1)}.
\end{eqnarray*}
When $s_i=0$ this reduces to
$$
\alpha_i(s) =\frac{b}{b+1-a}
$$
and when $s_i=1$ it reduces to
$$
\alpha_i(s) =\frac{1-b}{1-b+a}.
$$
Now, let $\mathcal I_s$ denote the set of all boxes that are tagged under the optimal policy when $\pi'=\pi'_*$ and the observed signal is $S=s$. The expected reward from following this policy is
\begin{eqnarray*}
\sum_{s\in\Sigma} \sum_{i\in\mathcal I_s} R_i(s) p(s) = \sum_{s\in\Sigma} \sum_{i\in\mathcal I_s} R_i(s) \prod_{j=1}^n P(S_i=s_i).
\end{eqnarray*}

\section{Conclusions}\label{sec:con}

In this work, we introduced the Test and Find (TF) model and derived an explicit solution. Since this solution can become computationally intensive at larger scales, we developed more efficient results for the symmetric case. In addition, we presented explicit solutions for three practically important extensions of the TF problem. There are a number of potential directions for future work. We briefly discuss a few:
\begin{itemize}
\item An interesting version of  the TF problem could be one where there are several subgroups. Within each subgroup the testing parameters and the prior probabilities are the same, but these are allowed to vary between different subgroups. We expect that one should be able to get computationally efficient results, even for larger-scale problems, so long as the number of subgroups is relatively small. This situation can be seen as an intermediate point between the general TF model and the symmetric case.
\item One could consider the situation where multiple objects may be placed in the same box and the DM needs to guess not only which boxes have objects, but the number of objects in each box. Here, the testing could give information about the number of objects in the box and the probabilities of false positives and false negatives could depend on this number.
\item It would be interesting to study the case where there are  three kinds of tags: ``object is here,'' ``object is not here,'' and ``not sure," and with a correspondingly more general cost structure. Going further, there may be more types of tags, each reflecting a different level of confidence. 
\item One could consider a layered TF problem, where each box contains multiple sub-boxes and we must identify not just the box, but the sub-box that contains the object. As a practical example, consider searching for a particular criminal in a city. First, based on a variety of information, the police must decide on which neighborhoods to focus on, and then they need to use potentially noisier (i.e., higher false positive and false negative rates) information to decide on which specific blocks in those neighborhoods to target.
\item Another extension could be that, after selecting the boxes, one would still need to search to see if the items are in these boxes. If the item is not found after some time, it may suggest that resources would be better spent searching boxes that were not originally selected. This would be a dynamic version of the TF problem in the spirit of the papers  \cite{Arkin:1964a} and \cite{Arkin:1964}. 
\item If some parameters are unknown, one can consider an iterated version of the TF problem, where the game is repeated multiple times and the DM aims to learn about the parameters while simultaneously trying to maximize the reward. Such problems are closely related to the classical multi-armed bandit problem, see, e.g., \cite{prso90} for a monographic treatment.
\end{itemize}

\end{document}